\documentclass{amsart}
\usepackage{eurosym}
\usepackage{amsmath}
\usepackage{amsfonts}
\usepackage{graphicx}
\usepackage{caption}
\usepackage{subcaption}
\usepackage{verbatim}

\newtheorem{theorem}{Theorem}

\newtheorem{definition}[theorem]{Definition}

\newtheorem{lemma}[theorem]{Lemma}

\newtheorem{proposition}[theorem]{Proposition}
\newtheorem{remark}[theorem]{Remark}

\begin{document}

\title[Rigorous estimate on convergence to equilibrium and escape rates]{An elementary way to rigorously estimate convergence to equilibrium
and escape rates.}

\author{Stefano Galatolo$^1$}
\email{$^1$ galatolo@dm.unipi.it}
\address{Dipartimento di Matematica, Universita di Pisa, Via \ Buonarroti 1,Pisa - Italy}
\author{Isaia Nisoli$^2$}
\email{$^2$ nisoli@im.ufrj.br}
\address{Instituto de Matem\'{a}tica - UFRJ
 Av. Athos da Silveira Ramos 149,
Centro de Tecnologia - Bloco C
Cidade Universitária -
Ilha do Fund\~{a}o.
Caixa Postal 68530
21941-909 Rio de Janeiro - RJ - Brasil}
\author{Beno\^{\i}t Saussol$^3$}
\email{$^3$ benoit.saussol@univ-brest.fr}
\address{Laboratoire de Math\'{e}matiques de Brest UMR CNRS 6205
D\'{e}partement de Math\'{e}matiques
Universit\'{e} de Bretagne Occidentale
6, avenue Victor Le Gorgeu, CS 93837, F-29238 BREST Cedex 3 - France}

\begin{abstract}
We show an elementary method to obtain (finite time and asymptotic) computer  assisted explicit upper bounds on convergence to equilibrium (decay of correlations) and escape rate for systems satisfying a Lasota
Yorke inequality. The bounds are deduced by the ones of suitable approximations of the system's transfer operator.
  We also present some rigorous experiment on some nontrivial example.
\end{abstract}

\maketitle

\tableofcontents

\section{Introduction}

The evolution of a chaotic system is unpredictable and difficult to 
describe, but its statistical properties are sometime predictable and
(reasonably) simple to be described.
Many of these properties are related to its associated transfer operator, essentially via its spectrum.
This is a linear operator representing the action of the dynamics on
measures space. Let us consider the set $SM(X)$ of finite signed Borel measures 
on the metric space $X$. A Borel map $T\colon X\to X$ naturally induces a linear
operator $L_{T}:SM(X)\rightarrow SM(X)$ called the \textbf{transfer operator},
defined as follows. If $\mu \in SM(X)$ then $L_{T}[\mu ]\in SM(X)$
is the measure such that
\par
\begin{equation*}
L_{T}[\mu ](A)=\mu (T^{-1}(A)).
\end{equation*}
Sometimes, when no confusion arises, we simply denote $L_{T}$ by $L$.

In this paper we address two important statistical features of the dynamics: the rate of \emph{convergence to equilibrium} and the \emph{escape rate} in open dynamical systems. 

Convergence to equilibrium (see Section \ref{1}) is a quantitative
estimation of the speed in which a starting, absolutely continuous measure
approaches the physical invariant measure.
This is related to the spectrum of the transfer operator, since the speed is exponential in the presence of a spectral gap,
and to decay of correlations. Indeed, an upper bound on the decay of correlations can be obtained
from convergence to equilibrium estimation in a large class of cases, see 
\cite{AGP}). An estimation for these rates is a key step to deduce many other consequences:
central limit theorem, hitting times, recurrence rate... (see e.g. \cite{Bo,BS,G07,L2}).

The escape rate refers to open systems, where the phase space has a ``hole''
(see Section \ref{escape}) and one wants to understand
quantitatively, the speed of loss of mass of the system trough the hole.
This is related to the spectral radius of a truncated transfer operator and to the presence of metastable states (see e.g. the book \cite{BBF} for an introduction).

We will present a method which allows to obtain an efficient, effective and quite elementary {\em finite time} and {\em asymptotic} upper estimations for these decay rates. The strategy is applicable under two main assumptions on the transfer operator.

\begin{itemize}
\item the transfer operator is regularizing on a suitable space; it
satisfies a Lasota Yorke inequality (see Equation \ref{1});

\item the transfer operator can be approximated in a satisfactory way by a
finite dimensional one. (see Equation \ref{2}) 
\end{itemize}

The first item in some sense describes the small scale behavior of
the system. The regularizing action implies that at a small scale we see a kind of uniform behavior.
 The macroscopic behavior is then described by a ``finite resolution'' approximation of the transfer operator. This will be represented by suitable matrices, as the transfer operator is linear, and its main features will be computable from the coefficients of the matrix.

In the following we will enter in the details of how this strategy can be
implemented in general, and in some particular system for which we will
present some experiment rigorously implemented by interval arithmetics.

\subsubsection*{Acknowledgements}   BS thanks INdAM and the University of Pisa for support and hospitality. SG thanks the Laboratoire de Mathématiques de Bretagne Atlantique for support and hospitality.
IN would like to thank the University of Pisa for support and hospitality, the IM-UFRJ and CNPq.
This work was partially supported by the ANR project Perturbation (ANR-10-BLAN 0106) and by EU Marie-Curie IRSES Brazilian-European partnership in Dynamical Systems (FP7-PEOPLE-2012-IRSES 318999 BREUDS).

\section{Recursive convergence to equilibrium estimation for maps satisfying
a Lasota Yorke inequality\label{sec1}}

Consider two vector subspaces of the space of signed measures on $X$
$$\mathcal{ S\subseteq }\mathcal{ W\subseteq }SM(X),$$ 
endowed with two norms, the strong norm $||~||_s$ on $\mathcal{S}$ and the weak norm $||~||_{w}$ on $\mathcal{W}$, 
such that $||~||_s\geq ||~||_{w}$ on $\mathcal{W}$.

  We say that the probability measure preserving transformation $(X,T,\mu)$ has convergence to equilibrium with speed $\Phi$ with respect to these norms if for any Borel probability measure $\nu $ on $X$
\begin{equation}
||L^{n}\nu -\mu ||_{w}\leq ||\nu ||_{s}\Phi (n).  \label{wwe}
\end{equation}

This speed can be also estimated by the rate of convergence to $0$ of $||L^{n}\xi ||_{w}$ 
for signed measures $\xi \in \mathcal{S}$ such that $\xi (X)=0$  or  by estimating the decay rate of correlation integrals like  $$|\int f\circ T^{n}~g~d\mu-\int f~d\mu \int g~d\mu|$$ for observables $f,g$ in suitable function spaces.

It is important both to have a certified quantitative estimation for this convergence
at a given time (numerical purposes, rigorous computation of the invariant
measure as in \cite{GN2},\cite{H},\cite{I}, see also Remark \ref{cpmp}), or an estimation for its
asymptotic speed of convergence (computer assisted proofs of the speed of
decay of correlations and its statistical consequences).

In the literature, the problem of computing rigorous bounds on the decay of
correlation and convergence to equilibrium rate of a given system was approached
by  spectral stability results (see e.g. \cite{L}).
The use of these methods is limited by the complexity of the assumptions and of the a priori estimations which are needed.

 In the following we show how the Lasota Yorke inequality, which can be
established in many systems, coupled with a suitable approximation of the
system by a finite dimensional one allows directly to deduce finite time and asymptotic  upper bounds
on the convergence to equilibrium of the system in a simple and elementary way.

{\bf Assumptions.} Let us suppose that our system satisfies:

\begin{itemize}
\item The system satisfies a Lasota Yorke inequality.
There exists constants $A,B,\lambda _{1}\in {\mathbb R}$ and $\lambda _{1}<1$ such that 
$\forall f\in \mathcal{S},\forall n\geq 1$%
\begin{equation}
||L^{n}f||_s\leq A\lambda _{1}^{n}||f||_s+B||f||_{w}.  \label{1}
\end{equation}%

\item There exists a family of "simpler" transfer
operators $L_{\delta }$ approximating $L$  satisfying a certain
approximation inequality: there are constants $C,D$ such that $\forall g\in 
\mathcal{S},\forall n\geq 0$:%
\begin{equation}
||(L_{\delta }^{n}-L^{n})g||_{w}\leq \delta (C||g||_s +nD||g||_{w}).  \label{2}
\end{equation}

\item There exists $\delta>0$, $\lambda _{2}<1$ and $n_1$ such that, setting
$$V=\{\mu\in \mathcal{S}|\mu (X)=0\}$$ 
we have  $L_\delta (V)\subseteq V $ and
\begin{equation}
\forall v\in V,~||L_{\delta }^{n_{1}}(v)||_{w}\leq \lambda _{2}||v||_{w}.
\label{3}
\end{equation}
\end{itemize}

\begin{remark}
In the following we will consider examples of
systems satisfying such inequality, where $\mathcal{S}$ is the space of
measures having a bounded variation density, $||~||_s$ is the bounded
variation norm and $||~||_{w}$ is the $L^{1}$ one.

We also remark that condition (\ref{2}) is natural for approximating operators 
$L_{\delta }$ defined (as it is commonly used), by $\pi _{\delta }L\pi
_{\delta }$, where $\pi _{\delta }$ is a projection on a finite dimensional
space with suitable properties, see Section \ref{appp}, where it is shown how to obtain (\ref{2})  under these assumptions.

We remark that in Equation (\ref{3}) , the condition is supposed on $L_{\delta }^{n}$ which is
supposed to be simpler than $L^{n}$ (e.g.  a discretization with
a grid of size $\delta $ represented by a matrix, see Section \ref{appineq}) and its
properties might be checked by some feasible computation.

\end{remark}

Under the above conditions we can effectively estimate from above the convergence to equilibrium in the system
in terms of the matrix
$$M=\left( 
\begin{array}{cc}
A\lambda _{1}^{n_{1}} & B \\ 
\delta C & \delta n_{1}D+\lambda _{2}%
\end{array}%
\right) .$$

Since $M$ is positive, its largest eigenvalue is 
\begin{equation}
\rho =\frac{A\lambda _{1}^{n_{1}}+\delta n_{1}D+\lambda _{2}+\sqrt{(A\lambda
_{1}^{n_{1}}-\delta n_{1}D-\lambda _{2})^{2}+4\delta BC}}{2}.  \label{eq:rho}
\end{equation}
Let $(a,b)$ be the left eigenvector of the matrix $M$ associated to the eigenvalue $\rho$, normalized in a way that $a+b=1$.

The condition $\rho<1$ below implies that the powers of $M$ go to zero exponentially fast.
Note that the quantities $\delta BC,$ $\delta n_{1}D$ have a chance to be small when 
$\delta $ is small, but this is not guaranteed since the choice of $n_{1}$ depends on $\delta $.
If we consider the case of piecewise expanding maps, with $L_{\delta }$ being the
Ulam approximation of $L$, this is the case (see \cite{GN}, Theorem 12 ).
 
\begin{theorem}\label{prop1}
Under the previous assumptions  \ref{1}, \ref{2}, \ref{3}, if $\rho<1$ then 
for any $g\in V$, 

(i) the iterates of $L^ {in_1}(g)  $ are bounded by 
\begin{equation*}
 \left(
\begin{array}{c}
|| L^{i n_1} (g) ||_s \\ 
|| L^{i n_1} (g) ||_w 
\end{array}
\right)
\preceq M^i
 \left(
\begin{array}{c}
||  g ||_s \\ 
||  g ||_w  
\end{array}
\right)
\end{equation*}
Here $\preceq $ indicates the componentwise $\leq $ relation (both
coordinates are less or equal).

(ii) In particular we have 
\[
||L^{in_{1}}g||_s \leq (1/a)\rho ^{i}||g||_s,
\]
and 
\[
||L^{in_{1}}g||_{w}\leq (1/b)\rho ^{i}||g||_s.
\]

(iii) Finally $\forall k\in \mathbb{N}$ 
\begin{eqnarray}
\Vert L^{k}g\Vert_s  &\leq &(A/a+B/b)\rho
^{\left\lfloor \frac{k}{n_{1}}\right\rfloor }||g||_s.  \label{decayrate} \\
\Vert L^{k}g\Vert _{w} &\leq & (B/b) \rho ^{\left\lfloor \frac{k%
}{n_{1}}\right\rfloor }||g||_s. \label{decayrate_weak}
\end{eqnarray}%

\end{theorem}

\begin{proof}
(i) 
Let us consider $g_{0}\in V $ and  denote $g_{i+1}=L^{n_{1}}g_{i}.$
By assumption \ref{1} we have
\begin{equation*}
||L^{n_{1}}g_{i}||_s \leq A\lambda _{1}^{n_{1}}||g_{i}||_s  +B||g_{i}||_{w} 
\end{equation*}
Putting together the above assumptions \ref{2} and \ref{3} we get
\begin{equation}
\begin{split}
||L^{n_{1}}g_{i}||_{w} 
& \leq ||L_{\delta }^{n_{1}}g_{i}||_{w}+\delta (C||g_{i}||_s+n_{1}D||g_{i}||_{w}) \\
& \leq \lambda _{2}||g_{i}||_{w}+\delta (C||g_{i}||_s +n_{1}D||g_{i}||_{w}).
\end{split}
\label{kkk}
\end{equation}

Compacting these two inequalities into a vector notation, setting $v_{i}=\left( 
\begin{array}{c}
||g_{i}||_s  \\ 
||g_{i}||_{w}%
\end{array}%
\right) $ 
we get
\begin{equation}
v_{i+1}\preceq \left( 
\begin{array}{cc}
A\lambda _{1}^{n_{1}} & B \\ 
\delta C & \delta n_{1}D+\lambda _{2}%
\end{array}%
\right) v_{i}  \label{4}
\end{equation}%
 The relation $\preceq $ can be used because
the matrix is positive. This proves (i) by an immediate induction.

(ii) Let us introduce the $(a,b)$ balanced-norm as $||g||_{(a,b)}=a||g||_s +b||g||_{w}$. The first assertion gives 
\[
||L^{in_{1}}g||_{(a,b)}\leq (a,b)\cdot M^i\cdot \left( 
\begin{array}{c}
||g||_s  \\ 
||g||_{w}%
\end{array}%
\right) ,
\]%
hence
\[
||L^{in_{1}}g||_{(a,b)}\leq \rho ^{i}||g||_{(a,b)}.
\]%
{}From this, the statement follows directly.

(iii) Writing any integer $k=in_{1}+j$ with $0\leq j<n_{1}$ , by assumption \ref{1} we have
\[
\Vert L^{k}(g_{0})\Vert_s  \leq A\lambda _{1}^{j}\Vert L^{in_{1}}g_{0}\Vert _s
+B\Vert L^{in_{1}}g_{0}\Vert _{w}.
\]%
and the conclusion follows by the second assertion (ii).
\end{proof}

\begin{remark}
We  remark that our approach being based on a vector inequality has some
similarity with the technique proposed in \cite{H}. The first inequality used
is the same in both approaches, the second is different. Our inequality
relies on the approximation procedure and is more general. 
\end{remark}

\begin{remark} \label{cpmp}
We also remark that our method allows to bound  the strong norm of the iterates
of a zero average starting measure. Considering $L^{n}(m-\mu )$ where $m$ is a suitable starting measure  (Lebesgue measure in many cases) and $\mu $ is the invariant one, we can understand how
many iterations of $m$ are necessary to arrive at a given small (strong) distance from the
invariant one. This, added to a way to simulate iterations of $L$ with small
 errors in the strong norm, allows in principle the rigorous computation of the invariant measure up
to small errors in the strong norm.
\end{remark}

\section{Escape rates\label{escape}}

A system with a hole is a system where there is a subset $H$ such that when
a point falls in it, its dynamics stops there. We consider $H$ to be not
part of the set where the dynamics acts. Iterating a measure by the dynamics
will lead to loose some of the measure in the hole at each iteration,
letting the remaining measure decaying at a certain rate.

In many of such systems a Lasota Yorke and an approximation inequality
(Equations \ref{1}, \ref{2}) can be proved, hence the procedure of the previous section allows to
estimate this rate, which  is related to the spectral radius of the transfer operator.

Let us hence  consider a starting system without hole, with transfer operator $L$,
  consider an hole in the set $H$ and let $1_{H^{c}}$ be the indicator
function of the complement of $H$. The transfer operator of the system with
hole is given by  
\begin{equation*}
L_{H}f=1_{H^{c}}Lf.
\end{equation*}

Similarly to the convergence to equilibrium we can say that the escape rate of the system with respect to norms $||\ ||_s ,||\ ||_w $ is  faster than  $\Psi $ if   for each $f$:
\[ ||L^n f||_w  \leq \Psi(n) ||f||_s . \]

We will see in the next proposition, that an iterative procedure as the one of the previous section can be implemented to estimate the escape rate.
We will need  assumptions similar to the ones listed before.
These are natural assumptions for open  systems constructed from a system satisfying a Lasota Yorke inequality.
 We will verify it in some example of piecewise expanding maps with a hole (see Section \ref{holely}). 

\begin{theorem}\label{prop2}
With the notations of Theorem \ref{prop1} let  us suppose

\begin{itemize}

\item the system with hole satisfies for some $\lambda_1<1$
\begin{equation}
\forall n\ge 1,\quad ||L_{H}^{n}f||_s \leq A\lambda _{1}^{n}||f||_s +B||f||_{w}.
\end{equation}

\item $L_\delta ({\mathcal S})\subseteq {\mathcal S} $ and  there is an approximation inequality: $\forall g\in \mathcal{S}$%
\begin{equation}
\forall n\ge 1,\quad ||(L_{\delta }^{n}-L_{H}^{n})g||_{w}\leq \delta (C||g||_s +nD||g||_{w}).
\end{equation}
 
\item Moreover, let us suppose there exists $ n_{1}$ for which $A\lambda_1 ^{n_1} <1$ and $\lambda_2<1$ such that
\begin{equation}
\forall v\in \mathcal{S}, ||L_{\delta
}^{n_{1}}(v)||_{w}\leq \lambda _{2}||v||_{w}.
\end{equation}
\end{itemize}
Then for any $i\ge1$

\begin{equation*}
 \left(
\begin{array}{c}
|| L^{i n_1} (f) ||_s \\ 
|| L^{i n_1} (f) ||_w 
\end{array}
\right)
\preceq \left( 
\begin{array}{cc}
A\lambda _{1}^{n_{1}} & B \\ 
\delta C & \delta n_{1}D+\lambda _{2}
\end{array}
 \right)^i
 \left(
\begin{array}{c}
||  f ||_s \\ 
||  f ||_w  
\end{array}
\right).
\end{equation*}

 In particular, as before, we have 
\[
||L^{in_{1}}g||_s \leq (1/a)\rho ^{i}||g||_s,
\]
and 
\[
||L^{in_{1}}g||_{w}\leq (1/b)\rho ^{i}||g||_s.
\]

\end{theorem}

The proof of the theorem is essentially the same as the one of theorem \ref{prop1},
we remark  that the main difference between the two propositions is that now we are looking to the behavior of
iterates of $L_{H}$ and $L_{\delta }$ on the whole space $\mathcal{S}$  and not only on the
space of zero average measures $V$.

\section{The Ulam method}

We now give an example of a finite dimensional approximation for
the transfer operator which is useful in several cases: the Ulam method.

Let us briefly recall the basic notions. Let us suppose now that $X$ is a
manifold with boundary. The space $%
X $ is discretized by a partition $I_{\delta }$ (with $k$ elements) and the
transfer operator $L$ is approximated by a finite rank operator $L_{\delta } 
$ defined in the following way: let $F_{\delta }$ be the $\sigma -$algebra
associated to the partition $I_{\delta }$,  define the projection: $\pi
_{\delta }$ as $\pi _{\delta }(f)=\mathbf{E}(f|F_{\delta })$, then%
\begin{equation}
L_{\delta }(f):=\mathbf{\pi }_{\delta }L\pi _{\delta }f.  \label{000}
\end{equation}

In the literature it is shown from different points of views, that taking
finer and finer partitions, in suitable systems, the behavior of this finite dimensional approximation
converges in some sense to the behavior of the original system, including the convergence of the spectral picture, see e.g. \cite{FCMP, BH, F08, L, B, MFG, B}.

We remark that this approximation procedure satisfies the approximation assumption (\ref{2})
 if for example  Bounded Variation    and $L^{1}$ norms are considered, see Lemma \ref{lemp} .
This allows the effective use of this discretization in Theorems \ref{prop1} and \ref{prop2}  for the sudy of   piecewise expanding maps.

\section{Piecewise expanding maps and Lasota Yorke inequalities\label{appineq}}

Let us consider a class of maps which are locally expanding but they can be
discontinuous at some point. We recall some results, showing that these systems satisfy the assumptions needed in our main theorems, in particular the Lasota Yorke inequality.

\begin{definition}
A nonsingular function $T:([0,1],m)\rightarrow ([0,1],m)$ is said to be piecewise
expanding if

\begin{itemize}
\item There is a finite set of points $d_{1}=0,d_{2},...,d_{n}=1$ such that $%
T|_{(d_{i},d_{i+1})}$ is $C^{2}$.

\item $\inf_{x\in \lbrack 0,1]}|D_{x}T|=\lambda _{1}^{-1}>2$ on the set
where it is defined.
\end{itemize}
\end{definition}

It is now well known (see e.g. \cite{LY}) that this kind of maps have an
a.c.i.m. with bounded variation density.

\subsection{Lasota Yorke inequality}

Let us consider an absolutely continuous measure $\mu$ and the following norm

\begin{equation*}
||\mu ||_{BV}=\underset{\phi \in C^{1},|\phi |_{\infty }=1}{\sup |\mu (\phi
^{\prime })|}
\end{equation*}
this is related to the classical definition of bounded variation: it is
straightforward to see that if $\mu $ has density $f$ then
\footnote{ If on a interval $I$, $f\geq q$ then consider a function $\phi \in C^{1}$ which
is =-1 on the left of the interval and =1 on the right, and increasing inside, then $\phi^{\prime }\geq 0$ , $\int \phi ^{\prime }dx=2$ and $\int f\phi ^{\prime
}dx\geq 2q$ we can do similarly if $f\leq -q,$\ hence $||\mu ||\geq 2||f||_{\infty }$. The existence of such interval $I$ of course cannot be ensured in general. In this case the use of Lebesgue's density theorem allows to find an interval where the same argument applies up to a small error.  } 
$2||f||_{\infty }\leq ||\mu ||_{BV}$.

The following inequality can be established (see e.g.  \cite{GN}, \cite{L2}) for the transfer operator of piecewise expanding maps.

\begin{proposition}\label{theo:LY}
\label{th8} If $T$ is piecewise expanding as above and $\mu $ is a measure
on $[0,1]$ 
\begin{equation*}
||L\mu ||_{BV}\leq \frac{2}{_{\inf T^{\prime }}}||\mu ||_{BV}+\frac{2}{\min
(d_{i}-d_{i+1})}\mu (1)+2\mu (|\frac{T^{\prime \prime }}{(T^{\prime })^{2}}%
|).
\end{equation*}
\end{proposition}

We remark that, if an inequality of the following form 
\begin{equation*}
||Lg||\leq 2\lambda ||g||+B^{\prime }|g|_{w}
\end{equation*}%
is established (with $2\lambda <1$) then,  iterating we have $%
||L^{2}g||\leq 2\lambda ||Lg||+B^{\prime }|g|_{w}=2\lambda (2\lambda
||Lg||+B^{\prime }|g|_{w})+B^{\prime }|g|_{w}...$ and thus%
\begin{equation*}
||L^{n}g||\leq 2^{n}\lambda ^{n}||Lg||+\frac{B^{\prime }}{1-2\lambda }%
|g|_{w}.
\end{equation*}

\subsection{Piecewise expanding maps with holes}\label{holely}

Let us consider a piecewise expanding map with a hole, which is an interval $I$. 
Let us see that if the system without hole  satisfies a Lasota Yorke inequality and is sufficiently expanding, we can
deduce a Lasota Yorke inequality for the system with an hole.

Let indeed consider a starting measure with bounded variation density, $f$.

If the piecewise expanding map without hole satisfies

\begin{equation*}
||Lg||_{BV}\leq 2\lambda ||g||_{BV}+B^{\prime }||g||_{1},
\end{equation*}

then the action of the hole is to multiply $Lg$ by $1_{I^{c}}$ hence
introducing two new jumps which we can bound by $||Lf||_{\infty }\leq \frac{1%
}{2}||Lf||_{BV}$, hence 
\begin{equation*}
||L_{H}g||_{BV}\leq 2||Lg||_{BV}
\end{equation*}%
and the system will satisfy%
\begin{equation*}
||L_{H}g||_{BV}\leq 4\lambda ||g||_{BV}+2B^{\prime }||g||_{1}.
\end{equation*}

This is enough to obtain the Lasota Yorke inequality in interesting examples
of piecewise expanding maps with holes (see Section \ref{escexp} ).

We recall that considering the Ulam discretization for these systems, the needed approximation inequality follows again, by Lemma \ref{lemp},  and the assumptions of Section \ref{escape} applies.

\begin{remark}
We remark that for our method to be applied no conditions on the size of the hole are needed.
\end{remark}

\section{Numerical experiments}

\subsection{Convergence to equilibrium: Lanford map}\label{subsec:Lanford}
In this subsection we estimate the speed of convergence to equilibrium for the  map which was
investigated in \cite{Lan}. The map $T:[0,1]\rightarrow \lbrack 0,1]$ is
given by 
\begin{equation*}
T:x\mapsto 2x+\frac{1}{2}x(1-x)\quad (\text{mod }1).
\end{equation*}%
To apply Proposition \ref{theo:LY} we have to take into account its second iterate  $F:=T^{2}$. 

The data below refers to the map $F$ and to the discretization of its transfer operator, 
and are some input and outputs of our algorithm.

\begin{center}
\begin{equation*}
\begin{array}{cc}
A=1 & \lambda_1\leq 0.32 \\
B\leq 30.6 & \delta=1/1048576 \\
\lambda_2<0.5 & n_1=18 \\
C \in [1.94,1.95] & D \in [70.992,70.993]
\end{array}
\end{equation*}
\end{center}

\begin{remark}
Using the estimates developed in \cite{GN}, an approximation $f_{\delta}$  for the invariant density $f$  can be computed with an explicit bound on the error.  The graph of the map and the computed density are drawn in figure \ref{fig:lanford} and \ref{fig:lanfordinvariant} respectively.  In the plotted case  $||f-f_{\delta}||_1\leq 0.016$.
\end{remark}

\begin{figure}[!h]
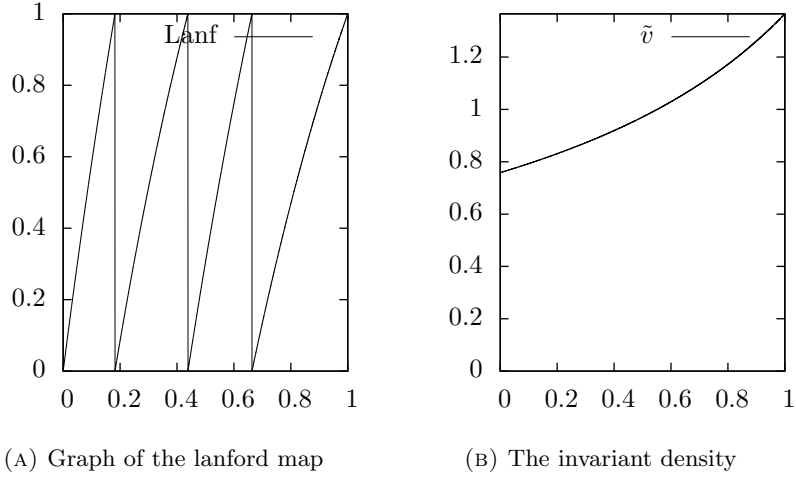

  \begin{subfigure}[b]{0.45\textwidth}
 \input{lanfdynamic.tex}
  \caption{Graph of the lanford map}
  \label{fig:lanford}
  \end{subfigure}  
  \begin{subfigure}[b]{0.45\textwidth}
 \input{lanfdensity.tex}
 \caption{The invariant density}
    \label{fig:lanfordinvariant}
\end{subfigure}
\caption{The Lanford map}
\end{figure}

The matrix $M$ that corresponds to our data is such that
\[M \preceq \left[\begin{smallmatrix} 1.24\cdot 10^{-9} & 30.6\\ 
1.86\cdot 10^{-6} & 0.5013 \end{smallmatrix}\right].\]

Using equation \eqref{eq:rho} we can compute:
\[\rho \in [0.5013,0.5014].\]

\begin{remark}\label{rem:coeffab}
With a simple computation it is possible to see that the coefficients $a,b$ (see Theorem \ref{prop1}) associated to the leading eigenvalue $\rho$ are:
\begin{align*}
a&=\frac{A\lambda _{1}^{n_{1}}-\lambda_2-\delta n_1 D+\sqrt{(A\lambda
_{1}^{n_{1}}-\delta n_{1}D-\lambda _{2})^{2}+4\delta BC}}{A\lambda _{1}^{n_{1}}-\lambda_2-\delta n_1 D+2B+\sqrt{(A\lambda
_{1}^{n_{1}}-\delta n_{1}D-\lambda _{2})^{2}+4\delta BC}}\\& b=\frac{2B}{A\lambda _{1}^{n_{1}}-\lambda_2-\delta n_1 D+2B+\sqrt{(A\lambda
_{1}^{n_{1}}-\delta n_{1}D-\lambda _{2})^{2}+4\delta BC}}.
\end{align*}
\end{remark}

In this example, hence:
\[a\in [3.6,3.7]\cdot 10^{-6},\quad b\in [0.9999963,0.9999964]\]

Therefore,if $L$ is the transfer operator associated to $F=T^2$ and using \eqref{decayrate} and \eqref{decayrate_weak}, we have the following
estimates:
\begin{align*}
\Vert L^{k}g\Vert_{BV}  &\leq (270839)\cdot (0.5014)
^{\left\lfloor \frac{k}{18}\right\rfloor }||g||_{BV}.  \label{decayrate} \\
\Vert L^{k}g\Vert _{L^1} &\leq  (30.7) \cdot (0.5014)
^{\left\lfloor \frac{k}{18}\right\rfloor }||g||_{BV}.
\end{align*}

We can also use the coefficients of the powers of the matrix (computed using interval arithmetics) to obtain upper bounds 
as in the following table:
\begin{align*}
\begin{array}{ccc}
\textrm{iterations} & \textrm{bound for }||L^h g||_{BV} & \textrm{bound for }|| L^h g||_1 \\
h=36 & 5.665\cdot 10^{-5}||g||_{BV}+15.34||g||_1 & 9.279\cdot 10^{-7}||g||_{BV}+2.513\cdot 10^{-1}||g||_1 \\
h=72 & 1.424\cdot 10^{-5}||g||_{BV}+3.855||g||_1 & 2.333\cdot 10^{-7}||g||_{BV}+6.316\cdot 10^{-2}||g||_1 \\
h=108 & 3.578\cdot 10^{-6}||g||_{BV}+9.689\cdot 10^{-1}||g||_1 & 5.862\cdot 10^{-8}||g||_{BV}+1.588\cdot 10^{-2}||g||_1 \\
h=144 & 8.992\cdot 10^{-7}||g||_{BV}+2.436\cdot 10^{-1}||g||_1 & 1.474\cdot 10^{-8}||g||_{BV}+3.990\cdot 10^{-3}||g||_1 \\
\end{array} 
\end{align*}

\subsection{Convergence to equilibrium: a Lorenz-type map}
In this subsection we estimate the speed of convergence to equilibrium for a Lorenz type $1$-dimensional map
using the recursive estimate established in section \ref{sec1}.
To do so, we use the estimates and the software developed in \cite{GN,GN2}.

The subject of our investigation is the $1$ dimensional Lorenz map acting on $I=[0,1]$ given by:
\[
T(x)=\left\{
\begin{array}{cc}
\theta \cdot |x-1/2|^{\alpha} & 0\leq x< 1/2 \\
1-\theta \cdot |x-1/2|^{\alpha} & 1/2 < x \leq 1
\end{array}
\right.
\]
with $\alpha=57/64$ and $\theta=109/64$.

Please note that since the Lorenz $1$-dimensional map does not have bounded derivative, the application of Proposition \ref{theo:LY}, 
is not immediate, but the following is true.

\begin{theorem}[\cite{GN2} Theorem 21]
Let $T$ be a $1$-dimensional piecewise expanding map, possibly with infinite derivative. 
Denote by $\{d_i\}$ the set of the discontinuity points of $T$, increasingly ordered. 
Fixed a parameter $l>0$ let \[I_l=\bigg\{x\in I \mid \bigg|\frac{T''(x)}{(T'(x))^2}\bigg|>l\bigg\}.\]
We have that $T$ satisfies a Lasota Yorke inequality
\[
||L\mu ||_{BV}\leq \lambda_1 ||f||_{BV}+B||f||_1.
\]
with
\[
\lambda_1\leq \frac{1}{2}\int_{I_l}\bigg|\frac{T''(x)}{(T'(x))^2}\bigg|dx+\frac{2}{\inf |T'|} \quad B\leq \frac{1}{1-\lambda_1}\bigg(\frac{2}{\min(d_{i+1}-d_i)}+l\bigg)
\]
(and in particular  it is possible to choose $l$ in a way that $\lambda_1 < 1$ ).
\end{theorem}

We apply our strategy to the map $F:=T^4$, please remark that, once we have the Lasota-Yorke coefficients and the discretization, the approach
is the same as in Subsection \ref{subsec:Lanford}.
The data below refers to the map $F$ and to the discretization of its transfer operator, 
and are some input and outputs of our algorithm; to compute the Lasota-Yorke constants we fixed $l=300$.
\begin{center}
\begin{equation}\label{tab:lorenz}
\begin{array}{cc}
A=1  &  \lambda_1\leq 0.884\\
B\leq 4049 & \delta=1/2097152 \\
\lambda_2<0.002 & n_1=10 \\
C\in [16.24,16.25] & D\in [11677.3,11677.4]
\end{array}
\end{equation}
\end{center}

\begin{remark}
The graph of the map and its  invariant density are drawn in figure \ref{fig:lorenz} and \ref{fig:lorenzinvariant} respectively.
Using the estimates developed in \cite{GN}, if $f$ is the density of the invariant measure and $f_{\delta}$ is the computed density, 
we have in this case that $||f-f_{\delta}||_1\leq 0.047$.
\end{remark}

\begin{figure}[!h]
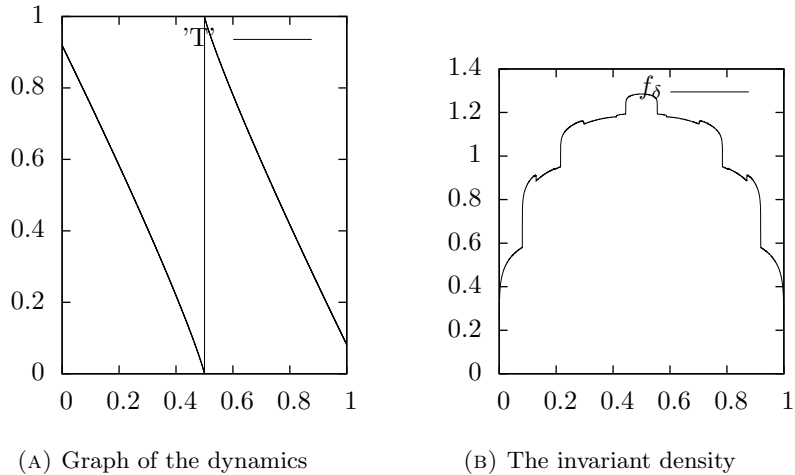

  \begin{subfigure}[b]{0.45\textwidth}
 \input{dynamic_lorenz.tex}
  \caption{Graph of the dynamics}
  \label{fig:lorenz}

  \end{subfigure}  
  \begin{subfigure}[b]{0.45\textwidth}
 \input{lorenz_invariant.tex}
 \caption{The invariant density}
    \label{fig:lorenzinvariant}
\end{subfigure}
\caption{Lorenz $1$-dimensional map}
\end{figure}

The matrix that corresponds to our data is such that
\[M \preceq \left[\begin{smallmatrix} 0.2915 & 4049\\ 
7.75\cdot 10^{-8} & 0.058 \end{smallmatrix}\right].\]

\begin{remark}
Please note that the Lorenz map we studied has some features that imply a large coefficient $B$ in the Lasota-Yorke inequality therefore
leading to a slower convergence to the equilibrium in the $||\ldotp||_{BV}$ norm.
\end{remark}

Using equation \eqref{eq:rho} we can compute:
\[\rho \in [0.386,0.387].\]

Using remark \ref{rem:coeffab} we have that the coefficients $a,b$ are such that:
\[a\in ([8.12,8.13]\cdot 10^{-5})\quad b\in [0.9999187,0.9999188]\]

Therefore,if $L$ is the transfer operator associated to $F=T^4$ and using \eqref{decayrate} and \eqref{decayrate_weak}, we have the following
estimates:
\begin{align*}
\Vert L^{k}g\Vert_{BV}  &\leq (16356)\cdot (0.387)
^{\left\lfloor \frac{k}{10}\right\rfloor }||g||_{BV}.  \label{decayrate} \\
\Vert L^{k}g\Vert _{L^1} &\leq  (4050) \cdot (0.387)
^{\left\lfloor \frac{k}{10}\right\rfloor }||g||_{BV}.
\end{align*}

We can also use the coefficients of the powers of the matrix (computed using interval arithmetics) to obtain upper bounds 
as in the following table:
\begin{align*}
\begin{array}{ccc}
\textrm{iterations} & \textrm{bound for }||L^h g||_{BV} & \textrm{bound for }|| L^h g||_1 \\
h=20 & 1.163\cdot 10^{-1}||g||_{BV}+1.414\cdot 10^{3}||g||_1 & 2.704\cdot 10^{-6}||g||_{BV}+3.469\cdot 10^{-2}||g||_1 \\
h=40 & 1.735\cdot 10^{-2}||g||_{BV}+2.134\cdot 10^{2}||g||_1 & 4.082\cdot 10^{-7}||g||_{BV}+5.025\cdot 10^{-3}||g||_1 \\
h=60 & 2.594\cdot 10^{-3}||g||_{BV}+31.92||g||_1 & 6.105\cdot 10^{-8}||g||_{BV}+7.513\cdot 10^{-4}||g||_1 \\
h=80 & 3.880\cdot 10^{-4}||g||_{BV}+4.774||g||_1 & 9.131\cdot 10^{-9}||g||_{BV}+1.124\cdot 10^{-4}||g||_1 \\
\end{array} 
\end{align*}

\subsection{Escape rates}\label{escexp}

In this subsection we estimate the escape rates for a non markov map using the estimates developed in Section\ref{sec1}.
We will study:
\[T(x)=\frac{23}{5}x \quad \textrm{mod } 1,\]
with an hole of size $1/8$ centered in $1/2$.

For the system with holes, the Lasota-Yorke inequality  has coefficients
\[A=1 \quad \lambda_1< 20/23 \quad B'\leq 7.08.\]

Below, some of the data (input and outputs) of our algorithm.
\begin{center}
\[
\begin{array}{cc}
A=1  &  \lambda_1\leq 20/23\leq 0.87\\
B\leq 7.08 & \delta=1/65536\\
\lambda_2<0.5 & n_1=17 \\
C \in [14.384,14.385]& D \in [20.31,20.32] 
\end{array}
\]
\end{center}

The graph of the map and a non rigorous estimation of the conditionally invariant density are drawn in figure \ref{fig:lin_nonmark} and \ref{fig:lin_nonmark_invariant} respectively.

\begin{figure}[!h]
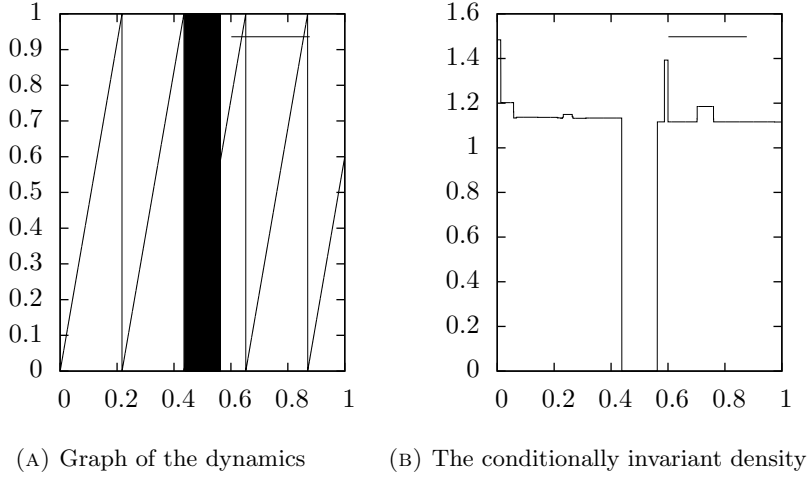

  \begin{subfigure}[b]{0.45\textwidth}
  \input{dynamic_hole.tex}
  \caption{Graph of the dynamics}
  \label{fig:lin_nonmark}

  \end{subfigure}  
  \begin{subfigure}[b]{0.45\textwidth}
\input{hole_invariant.tex}
 \caption{The conditionally invariant density}
    \label{fig:lin_nonmark_invariant}
\end{subfigure}
\caption{Piecewise linear non-markov map}
\end{figure}

The matrix that corresponds to our data is such that
\[M \preceq \left[\begin{smallmatrix} 0.094 & 7.08\\ 
2.1\cdot 10^{-4} & 0.506 \end{smallmatrix}\right].\]

Using equation \eqref{eq:rho} we can compute:
\[\rho \in [0.509,0.5091].\]

Using remark \ref{rem:coeffab} we have that the coefficients $a,b$ are such that:
\[a\in ([5.282,5.283]\cdot 10^{-4})\quad b\in [0.9994717,0.9994718]\]

Therefore,if $L$ is the transfer operator associated to $T$ and using \eqref{decayrate} and \eqref{decayrate_weak}, we have the following
estimates:
\begin{align*}
\Vert L^{k}g\Vert_{BV}  &\leq (1900.2)\cdot (0.5091)
^{\left\lfloor \frac{k}{17}\right\rfloor }||g||_{BV}.  \label{decayrate} \\
\Vert L^{k}g\Vert _{L^1} &\leq  (7.09) \cdot (0.5091)
^{\left\lfloor \frac{k}{17}\right\rfloor }||g||_{BV}.
\end{align*}

We can also use the coefficients of the powers of the matrix (computed using interval arithmetics) to obtain upper bounds 
as in the following table:
\begin{align*}
\begin{array}{ccc}
\textrm{iterations} & \textrm{bound for }||L^h g||_{BV} & \textrm{bound for }|| L^h g||_1 \\
h=34 & 1.034\cdot 10^{-2}||g||_{BV}+4.241||g||_1 & 1.315\cdot 10^{-4}||g||_{BV}+2.569\cdot 10^{-1}||g||_1\\
h=68 & 6.645\cdot 10^{-4}||g||_{BV}+1.134||g||_1 & 3.513\cdot 10^{-5}||g||_{BV}+6.654\cdot 10^{-2}||g||_1\\
h=102 & 1.559\cdot 10^{-4}||g||_{BV}+2.939\cdot 10^{-1}||g||_1 & 9.111\cdot 10^{-6}||g||_{BV}+1.724\cdot 10^{-2}||g||_1\\
h=136 & 4.025\cdot 10^{-5}||g||_{BV}+7.615\cdot 10^{-2}||g||_1 & 2.361\cdot 10^{-6}||g||_{BV}+4.467\cdot 10^{-3}||g||_1
\end{array} 
\end{align*}

\section{Appendix: The approximation inequality from the Lasota Yorke one.\label{appp}}

In this section we see that the approximation inequality (\ref{2} ) directly follows
from the Lasota Yorke inequality and from natural assumptions on the
approximating operator (including the Ulam discretization, which is used in the experiments).
 Hence the approximation inequality (\ref{2} ) that is assumed in the paper is a natural assumption in a general class of system, and approximations.

\begin{lemma}
\label{lemp} Suppose $L$ satisfy a Lasota Yorke inequality (\ref{1}) and $L_{\delta } $ is defined by composing with a ''projection'' $\pi_\delta $ satisfying a certain approximation inequality:

\begin{itemize}
\item $L_{\delta }=\pi _{\delta }L\pi _{\delta }$ with $||\pi _{\delta
}v-v||_{w}\leq \delta ||v||_s$ for all $v\in \mathcal{S}$

\item $\pi _{\delta }$ and $L$ are weak contractions for the norm $||~||_{w}$
\end{itemize}
 then $\forall f\in\mathcal{S}$%
\begin{equation*}
||L^{n}f-L_{\delta }^{n}f||_{w}\leq \delta \frac{(A\lambda _{1}+1)A}{%
1-\lambda _{1}}||f||_s+\delta Bn(A\lambda _{1}+2)||f||_{w}
\end{equation*}

\end{lemma}

\begin{proof}
It holds 
\begin{equation*}
||(L-L_{\delta })f||_{w}\leq ||\pi _{\delta }L\pi _{\delta }f-\mathbf{\pi }%
_{\delta }Lf||_{w}+||\mathbf{\pi }_{\delta }Lf-Lf||_{w},
\end{equation*}%
but 
\begin{equation*}
\mathbf{\pi }_{\delta }L\pi _{\delta }f-\mathbf{\pi }_{\delta }Lf=\mathbf{%
\pi }_{\delta }L(\pi _{\delta }f-f).
\end{equation*}%
Since both $\pi _{\delta }$ and $L$ are weak contractions and since $||\pi
_{\delta }v-v||_{w}\leq \delta ||v||_s$ 
\begin{equation*}
||\mathbf{\pi }_{\delta }L(\pi _{\delta }f-f)||_{w}\leq ||\pi _{\delta
}f-f||_{w}\leq \delta ||f||_s.
\end{equation*}

On the other hand%
\begin{equation*}
||\mathbf{\pi }_{\delta }Lf-Lf||_{w}\leq \delta ||Lf||_s\leq \delta (A\lambda
_{1}||f||_s+B||f||_{w})
\end{equation*}

which gives 
\begin{equation}
||(L-L_{\delta })f||_{w}\leq \delta (A\lambda _{1}+1)||f||_s+\delta B||f||_{w}
\label{1iter}
\end{equation}

Now let us consider $(L_{\delta }^{n}-L^{n})f$. It holds%
\begin{eqnarray*}
||(L_{\delta }^{n}-L^{n})f||_{w} &\leq &\sum_{k=1}^{n}||L_{\delta
}^{n-k}(L_{\delta }-L)L^{k-1}f||_{w}\leq \sum_{k=1}^{n}||(L_{\delta
}-L)L^{k-1}f||_{w} \\
&\leq &\sum_{k=1}^{n}\delta (A\lambda _{1}+1)||L^{k-1}f||_s+\delta
B||L^{k-1}f||_{w} \\
&\leq &\delta \sum_{k=1}^{n}(A\lambda _{1}+1)(A\lambda
_{1}^{k-1}||f||_s+B||f||_{w})+B||f||_{w} \\
&\leq &\delta \frac{(A\lambda _{1}+1)A}{1-\lambda _{1}}||f||_s+\delta
Bn(A\lambda _{1}+2)||f||_{w}.
\end{eqnarray*}
\end{proof}

\begin{remark}
From the above Lemma  we have in this case that $C$ and $D$ in \eqref{kkk} are given by
\[ C=\frac{(A\lambda _{1}+1)A}{1-\lambda _{1}},\quad  D=B\cdot(A\lambda_{1}+2).\]
\end{remark}

\end{document}